\documentclass{amsart}
\usepackage{amsmath}
\usepackage{amssymb}
\usepackage{amsthm}
\usepackage{dsfont}
\usepackage{lineno}
\usepackage{subcaption}
\usepackage{graphicx}
\usepackage{multirow}
\usepackage{color}
\usepackage{xspace}
\usepackage{pdfsync}
\usepackage{hyperref} 
\usepackage{algorithmicx}
\usepackage{algpseudocode}
\usepackage{algorithm}
\usepackage{mathtools}
\usepackage{enumitem}
\usepackage{epstopdf}
\usepackage{relsize}
\usepackage{wrapfig}
\graphicspath{{Figures/}}
\usepackage{xcolor}
\usepackage{tikz}
\usetikzlibrary{math}
\usetikzlibrary{decorations.pathmorphing,decorations.pathreplacing}

\DeclarePairedDelimiter\ceil{\lceil}{\rceil}

\newcommand{\bq}{\begin{equation}}
\newcommand{\eq}{\end{equation}}
\newcommand{\R}{\mathbb{R}}
\newcommand{\Z}{\mathbb{Z}}
\newcommand{\N}{\mathbb{N}}
\newcommand{\abs}[1]{\left\vert#1\right\vert}
\newcommand{\norm}[1]{\left\vert#1\right\vert}

\newcommand{\G}{\mathcal{G}}
\newcommand{\bO}{\mathcal{O}}

\newcommand{\Dt}{\mathcal{D}}

\newcommand{\Sf}{\mathcal{S}}

\newcommand{\MA}{Monge-Amp\`ere\xspace}

\newcommand{\Omb}[0]{\overline{\Omega}}

\newcommand{\x}[0]{\times}

\newcommand{\pOm}[0]{\partial \Omega}
\newcommand{\ra}[0]{\rightarrow}

\newcommand{\maop}[0]{\det(D^2u(x))}

\algnewcommand{\LineComment}[1]{\State \(\triangleright\) #1}

\newtheorem{theorem}{Theorem}

\theoremstyle{lemma}
\newtheorem{lemma}[theorem]{Lemma}

\newtheorem{corollary}[theorem]{Corollary}

\newtheorem{definition}[theorem]{Definition}

\theoremstyle{remark}

\begin{document}

\title[]{Domain Decomposition Methods for the Monge-Amp\`ere equation}

\author[Y. Boubendir]{Yassine Boubendir}
\address{Department of Mathematical Sciences, New Jersey Institute of Technology, University Heights, Newark, NJ 07102}
\email{boubendi@njit.edu}

\author[J. Brusca]{Jake Brusca}
\address{Department of Mathematical Sciences, New Jersey Institute of Technology, University Heights, Newark, NJ 07102}
\email{jb327@njit.edu}

\author[B. Hamfeldt]{Brittany Froese Hamfeldt}
\address{Department of Mathematical Sciences, New Jersey Institute of Technology, University Heights, Newark, NJ 07102}
\email{bdfroese@njit.edu}

\author[T. Takahashi]{Tadanaga Takahashi}
\address{Department of Mathematical Sciences, New Jersey Institute of Technology, University Heights, Newark, NJ 07102}
\email{tt73@njit.edu}

\thanks{The first and fourth authors were partially supported by NSF DMS-1720014 and DMS-2011843.  The second and third authors were partially supported by NSF DMS-1751996. }
\begin{abstract}
We introduce a new overlapping Domain Decomposition Method (DDM) to solve the fully nonlinear Monge-Amp\`ere equation.  While DDMs have been extensively studied for linear problems, their application to fully nonlinear partial differential equations (PDE) remains limited in the literature.  To address this gap, we establish 
a proof of global convergence of these new iterative algorithms using a discrete comparison principle argument. Several numerical tests are performed to validate the convergence theorem. These numerical experiments involve examples of varying regularity. Computational experiments show that method is efficient, robust, and requires relatively few iterations to converge. The results reveal great potential for DDM methods to lead to highly efficient and parallelizable solvers for large-scale problems that are computationally intractable using existing solution methods.
\end{abstract}

\date{\today}    
\maketitle
\section{Introduction}\label{sec:intro}
The \MA equation is a fully nonlinear second-order elliptic partial differential equation given by
    \begin{equation}\label{eq:MA}
        \begin{cases}
        \maop = f(x), & x\in \Omega \\
        u \text{ is convex}.
        \end{cases}
    \end{equation}
This PDE arises in numerous applications including design of optical systems~\cite{RomijnOptics2021,GutierrezAnisotropic2019}, medical image registration~\cite{Angenent_MedicalImaging}, economics~\cite{chiappori2017multi,galichon2016optimal}, meteorology~\cite{cullen2006mathematical}, surface evolution~\cite{OsherSethian}, machine learning~\cite{lin2020gans}, geophysics~\cite{EngquistFroese_Wass}, and optimal transport~\cite{Villani}.
Recent years have seen a great deal of progress in the development and analysis of new discretizations of the \MA equation~\cite{FO_MATheory,feng2021narrow,mirebeau2016minimal,WanMA,ObermanWS,FOFiltered,benamou2016monotone,feng2017convergent,Hamfeldt_gaussian,Nochetto_MAConverge,DeanGlowinski,dutchgroup,FengNeilan,BrennerNeilanMA2D,HL_ThreeDimensions}.
These discretizations reduce the PDE to a large system of nonlinear algebraic equations.

Much less attention has been given to the efficient solution of the resulting nonlinear algebraic systems.  Newton's method is a common choice of solver.  However, in practice Newton's method often scales poorly with problem size, particularly in the presence of non-smooth solutions or a loss of uniform ellipticity~\cite{FO_MANum}. In fact, in three-dimensions existing methods can be prohibitively expensive even on fairly small problems~\cite{HL_ThreeDimensions}.  There is a clear need for the development of efficient, parallelizable solvers for the \MA equation if these numerical methods are to keep pace with the demands of current applications. 

In this work, we describe, analyze, and test an overlapping Domain Decomposition Method (DDM) for the \MA equation.  DDMs were originally introduced and analyzed as iterative methods at the continuous PDE level~\cite{Lions_ddm}.  The study of DDMs for linear PDEs is now a fairly mature field~\cite{quarteroni_domain_1999,EwingDDM,SmithDDM,Gander_Schwarz,boubendir_quasi-optimal_2012,bendali_non-overlapping_2006-1}.  More recently, DDMs have been introduced as nonlinear solvers~\cite{DryjaHackbush,Lui,SpiteriMiellou,TaiEspedal}, linear solvers within a Newton iteration~\cite{CaiDryja}, and preconditioners~\cite{DoleanGander,CaiKeyes,ChaoiquiGander,CaiLi}.  

Despite the great potential in this setting, essentially nothing is known about the use of DDMs for fully nonlinear second order elliptic equations.  A key additional challenge in this setting is that wide finite difference stencils are often required, which prevents the use of standard boundary conditions at the interfaces between subdomains. The iterative algorithm proposed in this paper effectively combines wide stencil approximations for local problems with an overlapping decomposition method.    {The resulting solver is robust and well-suited for industrial problems that involve domains with large size}. 

In this work, we begin the process of developing and analyzing DDMs for solving the \MA equation.  In particular, we  propose, analyze, and test a new overlapping DDM for the Dirichlet problem, which has the flavor of a nonlinear additive Schwarz method.
However, in order to accommodate wide finite difference stencils, we introduce a non-standard interface condition that allows neighboring subdomains to interact along a narrow strip instead of solely at the boundary.

 We exploit a discrete comparison principle in order to prove that this approach, combined with any monotone discretization of the \MA equation, will converge to the desired solution given any initial guess.  We also implement and thoroughly test this method using both smooth and non-classical solutions of the \MA equation.  Numerical experiments validate the convergence proof and indicate that even without optimization, the method is efficient and robust.  These results suggest that DDM methods have the potential to become the highly efficient and parallelizable solver that is needed by current large-scale applications. 
{All code is publicly available at \texttt{https://github.com/tt73/MA-DDM}. }

{{This paper is organized as follows. The first section is devoted to background, including the description of the Monge-Amp\`ere equation and the wide stencil schemes used to approximate this equation. Domain decomposition methods are introduced in section 2, where we explain how to set and adapt overlapping DDM to the Monge Monge-Amp\`ere equation in connection with the wide stencil approach. Following that, we prove the convergence of the algorithm in section 3 using a discrete comparison principle argument.  The validation of these methods is presented in section 4. Section 5 discusses conclusions and future work. }}

\section{{Monge-Amp\`ere equation}}\label{sec:background}
The \MA equation is an example of a second-order degenerate elliptic partial differential equation, which take{s} the general form
\bq\label{eq:elliptic}
F(x,u(x),\nabla u(x), D^2u(x)) = 0, \quad x \in \bar{\Omega}.
\eq

\begin{definition}[Degenerate Elliptic]
Let $\Omega\subset\R^n$ and denote by $\Sf^n$ the set of symmetric $n\x n$ matrices.  The operator $F:\Omb\x\R\x\R^n\x\Sf^n \to \R$ is said to be \emph{degenerate elliptic} if 
\begin{equation*}
    F(x,u,p,X) \leq F(x,v,p,Y)
\end{equation*}
whenever $u \leq v$ and $ X \succeq Y$.
\end{definition}
We note that the operator is defined on the closure of $\Omega$, and takes on the value of the relevant boundary conditions at $\pOm$. For the Dirichlet problem, which is the setting implemented in this article, the PDE operator at the boundary is defined as
\bq\label{eq:dirichletOp}
F(x,u(x),\nabla u(x),D^2u(x)) = u(x) - g(x), \quad x \in \partial\Omega.
\eq

In general, degenerate elliptic equations need not have classical solutions, and some notion of weak solution is required.  {The Aleksandrov solution provides a geometric interpretation in terms of the subgradient measure, which allows for very general right-hand sides, including measures that do not have an associated density~\cite{Gutierrez}. Though slightly less general, the viscosity solution has proved to be particularly useful for this class of equations~\cite{CIL}, and forms the foundation for most of the recently developed numerical convergence proofs for the \MA equation.} The idea of the viscosity solution is to use a maximum principle argument to pass derivatives onto smooth test functions that lie above or below the semi-continuous envelopes of the candidate weak solution.

\begin{definition}[Viscosity Solution]
A bounded upper (lower) semi-continuous function $u$ is a \emph{viscosity subsolution (supersolution)} of~\eqref{eq:elliptic} if for every $\phi \in C^2(\Omb)$, that whenever $u-\phi$ has a local maximum (minimum) at $x\in\Omb$, then
\begin{equation*}
    F_*^{(*)}(x,u(x),\nabla \phi(x),D^2\phi(x)) \leq (\geq) 0.
\end{equation*}A bounded function $u:\Omb \ra \R$ is a \emph{viscosity solution} of~\eqref{eq:elliptic} if $u^*(x)$ is a viscosity subsolution and $u_*(x)$ is a viscosity supersolution.
\end{definition}



\subsection{Approximation of Elliptic Equations}

A fruitful technique for numerically solving fully nonlinear elliptic equations involves finite difference schemes of the form
\begin{equation}\label{eq:scheme}
    F^{h}(x,u(x),u(x) - u(\cdot)) = 0
\end{equation}
defined on a finite set of discretization points $\G\subset\Omega$ with characteristic spacing of the grid points encoded in the parameter $h>0$.
Many key results on the convergence of finite difference methods to the viscosity solution of a degenerate elliptic PDE are based upon a set of criterion developed by Barles and Souganidis~\cite{BSnum}.

\begin{definition}[Consistency]\label{def:consist}
The scheme~\eqref{eq:scheme} is \emph{consistent} with~\eqref{eq:elliptic} if, for any test function $\phi \in C^{2,1}(\Omb)$ and $x \in \Omb$, we have
\begin{align}
    \limsup_{h\ra 0^+,y\ra x, \xi \ra 0}F^h(y,\phi(y)+\xi,\phi(y)-\phi(\cdot)) \leq F^*(x,\phi(x),\nabla\phi(x),D^2\phi(x))
    \\
    \liminf_{h\ra 0^+,y\ra x, \xi \ra 0}F^h(y,\phi(y)+\xi,\phi(y)-\phi(\cdot)) \geq F_*(x,\phi(x),\nabla\phi(x),D^2\phi(x)).
\end{align}
\end{definition}


\begin{definition}[Monotonicity]\label{def:mono}
The scheme~\eqref{eq:scheme} is \emph{monotone} if $F^h$ is a non decreasing function of its last two arguments. 
\end{definition}

\begin{definition}[Stability]\label{def:stab}
The scheme~\eqref{eq:scheme} is \emph{stable} if there exists some $M>0$, independent of $h$, such that every solution $u^h$ satisfies $||u^h||_{\infty} < M$.
\end{definition}

These simple concepts lead immediately to convergence of finite difference methods, provided the underlying PDE satisfies a strong comparison principle.
\begin{definition}[Comparison principle]\label{def:comparison}
The PDE operator $F(x,u,p,X)$ satisfies a strong comparison principle if, whenever $u$ is a viscosity subsolution and $v$ is a viscosity supersolution, $u \leq v$.
\end{definition}

\begin{theorem}[Convergence~\cite{BSnum}]\label{thrm:Conv}
Let $u$ be the unique viscosity solution of the PDE~\eqref{eq:elliptic}, where $F$ is a degenerate elliptic operator with a strong comparison principle.   Let $u^h$ be any solution of~\eqref{eq:scheme} where $F^h$ is a consistent, monotone, stable approximation scheme. Then $u^h$ converges uniformly to $u$ as $h\to0$.
\end{theorem}

A strong comparison principle has never been established for the \MA equation; in fact, there are settings where it is known to fail~\cite{Hamfeldt_gaussian,jensen20187}.  However, alternate techniques have been designed to show convergence of many different monotone schemes to the weak solution of the \MA equation~\cite{feng2017convergent,Hamfeldt_gaussian,Nochetto_MAConverge}.  Key to much of this analysis are the concepts of a continuous and proper scheme.

\begin{definition}[Continuous]\label{def:continuous}
The scheme~\eqref{eq:scheme} is \emph{continuous} if $F^h$ is continuous in its second and third arguments.
\end{definition}



\begin{definition}[Proper]
The scheme~\eqref{eq:scheme} is \emph{proper} if there exists some $C>0$ such that if $u \geq v$ then $F^h(x,u,p)-F^h(x,v,p) \geq C(u-v)$.
\end{definition}

Perhaps surprisingly, monotone and proper schemes satisfy a discrete form of the comparison principle, even if the underlying PDE does not have a comparison principle.  This is tremendously important for establishing the well-posedness of approximation schemes, and will play a critical role in the analysis of DDM in the present article.

\begin{theorem}[Discrete Comparison Principle~{\cite[Theorem~5]{ObermanEP}}]\label{thm:DCP}
Let $F$ be a proper monotone scheme.  Suppose that $F(x,u(x),u(x)-u(\cdot))\leq(<) F(x,v(x),v(x)-v(\cdot))$ for every $x  \in \G$.  Then $u \leq(<) v$ on $\G$. 
\end{theorem}

\begin{theorem}[Existence and uniqueness~{\cite[Theorem~8]{ObermanEP}}]\label{thm:exist}
Let $F$ be a continuous, proper, monotone scheme. Then $F(x,u(x),u(x)-u(\cdot)) = 0$ has a unique solution.
\end{theorem}

\subsection{Wide stencil schemes}
The past several years have seen great interest in the design of monotone approximation schemes for the \MA equation~\cite{benamou2014monotone,Bonnet_OTBC,feng2017convergent,FroeseMeshfreeEigs,FO_MATheory,mirebeau2015MA,Nochetto_MAConverge,ObermanEigenvalues}.  A common feature of these methods is that they rely on wide stencils instead of more traditional nearest neighbors finite difference schemes.  In fact, wide stencils are a necessary feature of a consistent, monotone scheme for a general degenerate elliptic equation~\cite{Kocan,MotzkinWasow}.

Monotone discretization of the \MA equation typically involves a reformulation involving a nonlinear combination of various linear elliptic operators.  For example,  the \MA operator can be represented by the product of the eigenvalues of the Hessian matrix.  In two-dimensions, these eigenvalues are given by the largest and smallest second directional derivatives~\cite{ObermanEigenvalues}:
\bq
\det(D^2u) = \left(\min\limits_{\nu\in\R^2}\frac{\partial^2u}{\partial\nu^2}\right)\left(\max\limits_{\nu\in\R^2}\frac{\partial^2u}{\partial\nu^2}\right).
\eq
Discretization involves approximating the min/max using a finite collection of directions  $\nu$.  For example, on a Cartesian grid, the min/max may be computed using all grid-aligned directions $\nu = (m,n)\in\Z^2$ that have a maximal stencil width $  w \in \N$. 
See Figure \ref{fig:stencil}. The second directional derivatives can then be discretized using centered differences:

\bq\label{eq:centered}
\frac{\partial^2u}{\partial\nu^2}(x) \approx \frac{u(x+h\nu) + u(x-h\nu)-2u(x)}{h^2\abs{\nu}^2}.
\eq

In order to preserve both consistency and monotonicity at points near the boundary of the domain, where the wide stencil would extend outside the domain, a typical approach is to over-resolve the domain boundary (Figure~\ref{fig:circle}).  This allows one to maintain the same angular resolution $d\theta$, though the centered difference~\eqref{eq:centered} needs to be replaced with a lower-order uncentered finite difference scheme.

\begin{figure}
    \centering
    \begin{subfigure}[]{0.5\textwidth}\includegraphics[width=\textwidth]{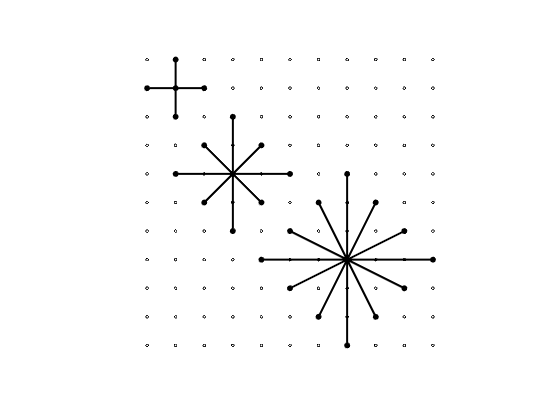}\caption{A wide finite difference stencil.}\label{fig:stencil}\end{subfigure}
    \begin{subfigure}[]{0.4\textwidth}\includegraphics[width=\textwidth]{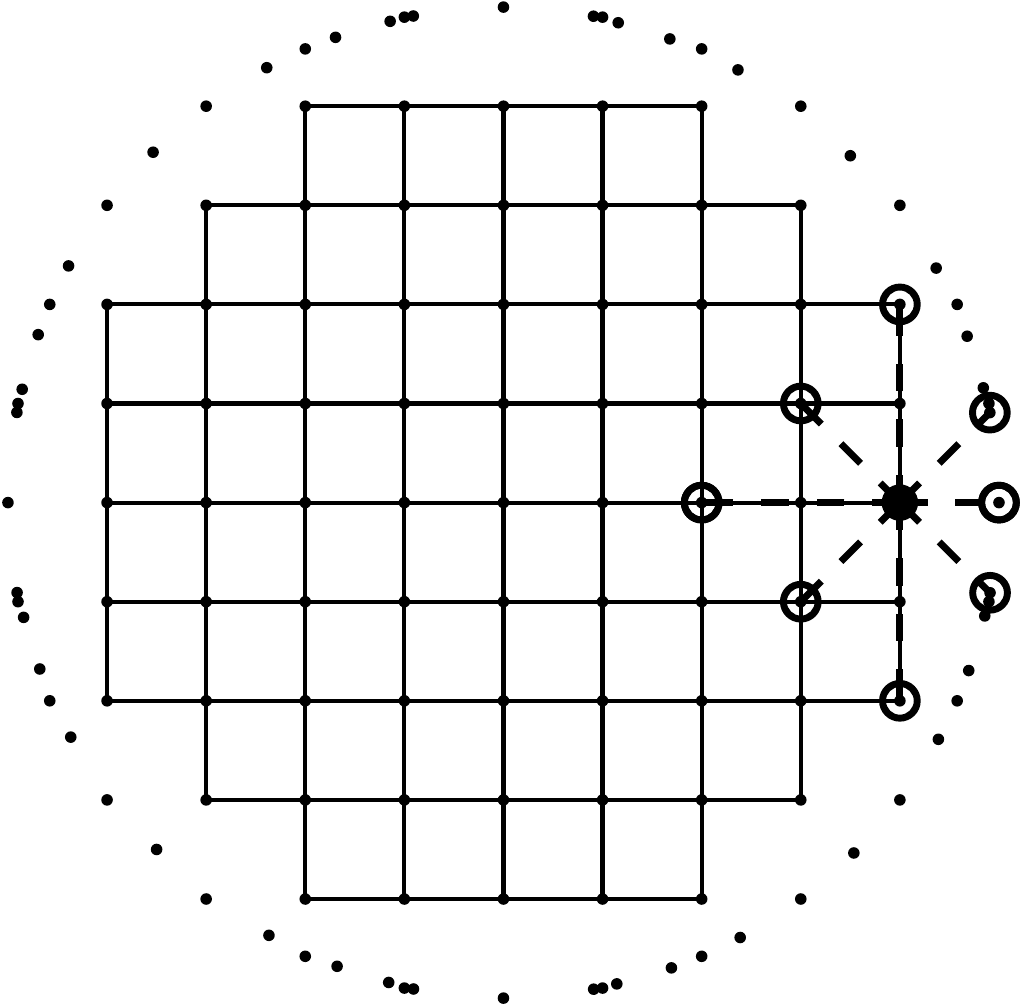}\caption{A discretization of a disc that preserves angular resolution up to the boundary.}\label{fig:circle}\end{subfigure}
    \caption{}
    \label{fig:ws}
\end{figure}

\section{Domain Decomposition}\label{sec:ddm}
In this section we introduce domain decomposition methods for nonlinear equations and show how they can be used as an iterative solver for the nonlinear discrete system.

\subsection{Motivating Example}
In order to approximate the solution to $F[u]=0$ we need to solve the nonlinear system of equations given by $F^h[u] = 0 $. In general, there are no direct methods to solve this; instead we must rely on an iterative method such as Newton's Method. Given some approximate solution $u^0$, we iterate 
\begin{align*}
    u^{n+1} = G(u^n,F^h[u^n],\nabla F^h[u^n])
\end{align*}
where $G$ encodes some root finding method which may depend on the Jacobian of the scheme $\nabla F^h$. In general, the update step can be expensive to form, and even more so to evaluate. {{In this paper, we use a new domain decomposition approach to design}} an improved iterative method. 

The methods discussed in this paper are focused on domain decomposition methods at the discrete level, but for simplicity we begin with an example at the continuous level.

We consider the Dirichlet problem for the \MA equation
\bq
    F(x,u) = \begin{cases}
         - \det(D^2u(x)) + f(x) = 0, &x\in \Omega,
        \\
        u(x) - g(x) = 0, &x \in \pOm\\
        u \text { is convex}
    \end{cases}
\eq
on a square domain $\Omega = (0,L)^2$ for some $L > 0$. We then decompose the the domain $\Omega$ into two overlapping rectangles 
\begin{align}
    \Omega_1 = (0,l_1)\times (0,L)
    \\
    \Omega_2 = (l_2,L) \times (0,L)
\end{align}
where $0 < l_2 < l_1 < L$. 

Our goal is to solve the PDE on each of these subdomains independently of each other, then combine the results into a global solution. However, in order to make each subdomain problem well-posed, we need to provide some additional boundary data on the portion of $\partial\Omega_i$ that lies within the interior of the global domain $\Omega$. 
We will refer to the sides where we need to impose boundary data as having artificial boundary data. In order to close the subdomain problems for computing the newest iteration, we impose artificial Dirichlet data using the value of the global solution approximation at the previous iterate.

The subdomain problems on $\Omega_1$ and $\Omega_2$ now become
\bq\label{eq:subdomain1cts}
    \begin{cases}
    -\det(D^2u_i^{n+1}(x)) + f(x) = 0, &x\in \Omega_i
    \\
    u_i^{n+1}(x)-g(x) = 0, &x \in \partial\Omega_i \cap \pOm
    \\
    u_i^{n+1}(x)-u^n(x) = 0, &x \in \{l_i\}\times (0,L)\\
    u_i^{n+1} \text{ is convex, } & x\in\Omega_i.
    \end{cases}
\eq
We denote the solutions of the subdomain problems at the $nth$ iterate as $u_1^n$ and $u_2^n$. Presently, each $u_i$ is only defined on $\Omega_i$. However, our goal is to combine these solutions into a new global approximation. For this reason, it will be more convenient to recast the problem so that the solutions take values on all of $\Omb.$ 

\bq\label{eq:subdomain2cts}
    F_i(x,u_i^n,u^n) = \begin{cases}
        -\det(D^2u_i^{n+1}(x)) + f(x) = 0, &x\in \Omega_i
        \\
        u_i^{n+1}(x)-g(x) = 0, &x \in \partial\Omega
        \\
        u_i^{n+1}(x)-u^n(x) = 0, &x \not\in \Omega_i \cup \pOm\\
        u_i^{n+1}(x) \text{ is convex}, & x \in \Omega_i.
    \end{cases}
\eq
These solutions are equivalent to the original subdomain problem~\eqref{eq:subdomain1cts}, but are extended to take on the values of $u^n$ outside the active subdomain.

To combine the subdomain solutions into an improved global approximation, we want to assign values of $u_i(x)$ when $x \in \Omega_i$. However, we note that in the overlap region $\Omega_1\cap\Omega_2$, the subdomain solutions $u_1$ and $u_2$ need not be equal.  Our approach is to take some average of these in the overlap region.  The resulting global approximation is given by
\bq\label{eq:globalcts}
u^{n+1}(x) = \begin{cases}
u_1^{n+1}(x), & x \in [0,l_2] \times[0,L]\\
\dfrac{u_1^{n+1}(x)+u_2^{n+1}(x)}{2}, & x \in (l_2,l_1) \times[0,L]\\
u_2^{n+1}(x), & x \in [l_1,L] \times[0,L].
\end{cases}
\eq




\subsection{Discrete DDM}
In the motivating example, $\Omega$ had a very simple geometry that was easily decomposed into two subdomains. In order to {{design}} a versatile method, we need to allow for more general domains and multiple subdomains. In addition, we need to apply the procedure at the discrete level, which poses additional challenges when using wide stencil schemes.


We begin with a discretization $\G\subset\bar{\Omega}$ of the domain with a grid resolution $h>0$.  We assume also a consistent, monotone, proper, continuous finite difference scheme of the form
\bq\label{eq:schemetosolve}
F^h(x,u(x),u(x)-u(\cdot)) = 0, \quad x \in \G.
\eq
We assume that this scheme incorporates the Dirichlet boundary data so that
\bq\label{eq:schemebc}
F^h(x,u(x),u(x)-u(\cdot)) = u(x)-g(x), \quad x \in \G\cap\partial\Omega.
\eq
Our goal now is to design a DDM that will converge to the unique solution $u^*$ of this discrete system of equations.

We begin by selecting any overlapping decomposition of the domain $\Omega$ into open sets
 $\Omega_1,...,\Omega_d$ such that
\[
    \Omega = \bigcup\limits_{i=1}^{N_d} \Omega_i
\]
and $N_d \in \N$ represents the number of distinct subdomains. 
This allows us to define a set of discrete subdomains $\G_1,...,\G_{N_d}$ given by 
\[
    \G_i = \G \cap \Omega_i.
\]

As in the continuous example, we first need to define subdomain problems.  A traditional overlapping DDM would impose artificial boundary data $v$ at grid points lying along the artificial boundary, leading to a subdomain problem of the form
\bq\label{eq:sub1}
\begin{cases}
F^h(x,u_i(x),u_i(x)-u_i(\cdot)) = 0, & x \in \G_i\\
u_i(x) - g(x) = 0, & x \in \G\cap\partial\Omega\\
u_i(x) - v(x) = 0, & x \in \G\cap\Omega\cap\partial\Omega_i.
\end{cases}
\eq
However, even with a careful choice of subdomains that ensures that the artificial boundaries are properly resolved by the grid $\G$, this will not lead to a well-defined problem.  There is an additional challenge caused by the wide finite difference stencils that are typically required for the \MA equation.  A consequence is that at points near the boundary of the subdomain $\Omega_i$, the finite difference stencils will draw on values lying beyond the boundary $\partial\Omega_i$.  See Figure~\ref{fig:subdomains}.

\begin{figure}
    \centering
		\includegraphics[width=0.8\textwidth]{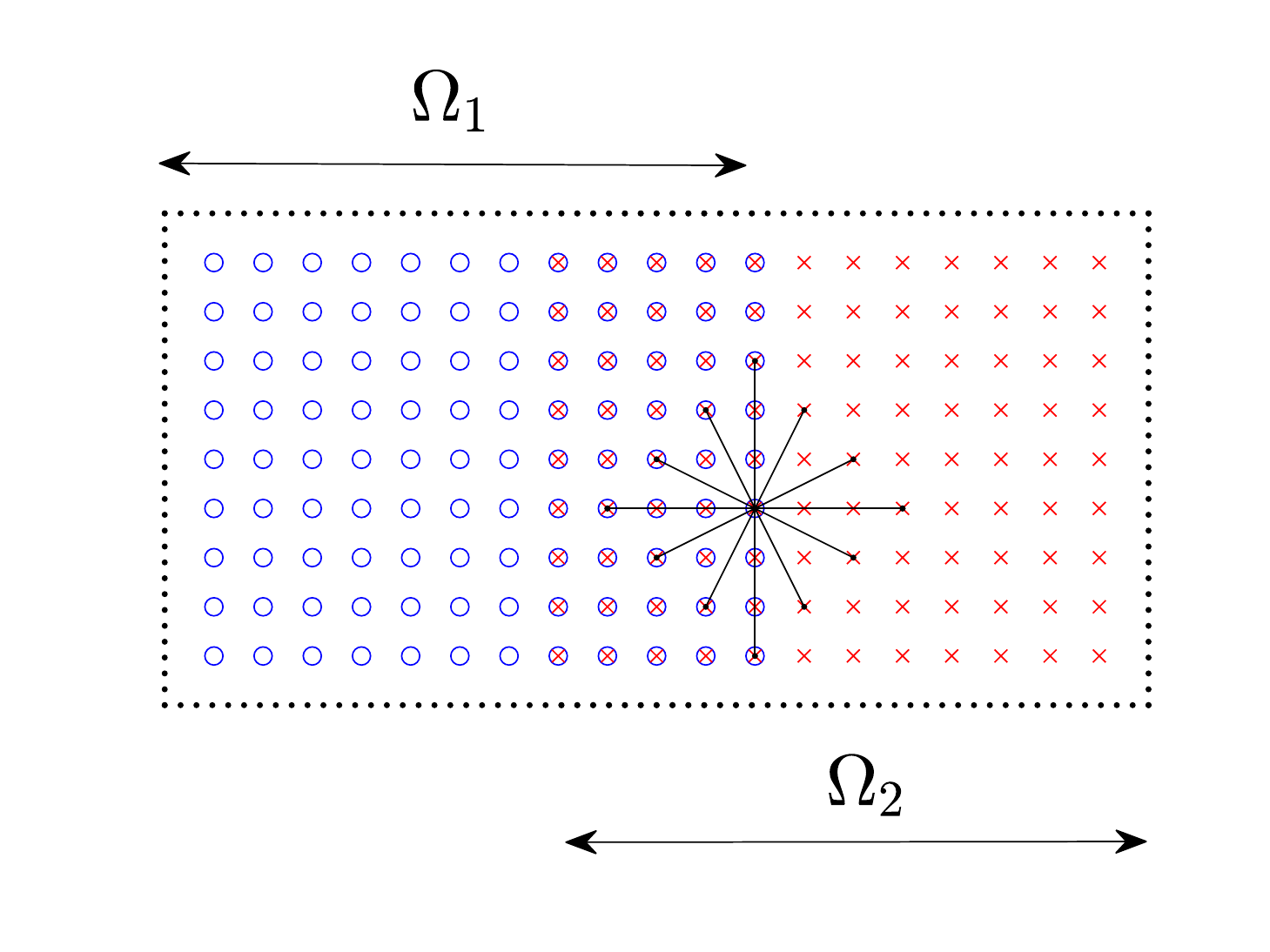}
    \caption{A two-subdomain decomposition of the domain $\Omega$. Points in $\G_1$ are denoted by $\circ$, points in $\G_2$ by $\times$, and points in $\G\cap\partial\Omega$ by $\cdot$.  The boundary is over-resolved to preserve both consistency and monotonicity of a wide stencil scheme.  Near the boundary of $\Omega_i$, the wide stencil may extend significantly into $\Omega_j$ ($j \neq i$), which requires the condition $u_i^{n+1}(x) = u^n(x)$ to be enforced along a strip of grid points in practice.}
    \label{fig:subdomains}
\end{figure}

In order to accommodate stencils of width $w$, it is necessary to provide data for $u_i^{n+1}$ at all grid points within a strip of width $wh$ neighboring the boundary $\Omega\cap\partial\Omega_i$.  This is equivalent to posing an appropriate subdomain problem within the entire computational domain $\G$ (as in~\eqref{eq:subdomain2cts} in the continuous example).  To this end, we define the following subdomain operators for $i = 1, \ldots, N_d$:
\bq\label{eq:sub2}
F^h_i(x,u_i;v) \equiv \begin{cases}
F^h(x,u_i(x),u_i(x)-u_i(\cdot)), & x \in \G_i\\
u_i(x) - g(x), & x \in \G\cap\partial\Omega\\
u_i(x) - v(x), & x \in \G\cap\Omega\cap\G_i^c.
\end{cases}
\eq

We observe that this subdomain operator inherits many of the properties of the original discrete approximation $F^h$; in particular, it is consistent, monotone, proper, and continuous.  This allows us to uniquely define a solution operator $S^h_i[v]$ (Theorem~\ref{thm:exist}) such that
\bq\label{eq:solnop}
F_i^h(x,S_i^h[v];v) = 0, \quad x \in \G
\eq
for any grid function $v$ defined on $\G$.


We can now define a DDM iteration.  Suppose that we are given an approximation $u^n:\G\to\R$.  We begin by defining $u_1^{n+1}, \ldots, u_{N_d}^{n+1}$ as the solutions to the $N_d$ subdomain problems:
\bq\label{eq:solvesub}
u_i^{n+1} = S_i^h[u^n], \quad i = 1, \ldots, N_d.
\eq
We recall again that this involves restricting the finite difference approximation~\eqref{eq:schemetosolve}  to the $ith$ subdomain, with any missing data supplied by the results of the previous iteration.

Once we have computed the subdomain solutions $u^{n+1}_1,...,u^{n+1}_{N_d}$, we need to combine them to create a global update $u^{n+1}$.  As in the continuous example, we wish to utilize the solution $u_i^{n+1}$ in the $ith$ subdomain.  However, once again we must account for the fact that some points $x\in\G$ will lie in multiple subdomains.  In these instances, we utilize a weighted average of all possible subdomain solutions. To accomplish this, we define a set of weights $\lambda_1(x),...,\lambda_{N_d}(x)$ on $\G$ with the following properties:
\bq\label{eq:lambs}
    \begin{cases}
        0 \leq \lambda_i(x) \leq 1 , &x \in \G, \,\, i = 1,...,N_d
        \\
        \lambda_i(x) > 0, &x \in \G_i, \,\, i = 1,...,N_d
        \\
        \lambda_i(x) = 0, &x \not \in \G_i, \,\, i = 1,...,N_d
        \\
        \sum\limits_{i = 1}^{N_d} \lambda_i(x) = 1, & x \in \G.
    \end{cases}
\eq

Now we define the update $u^{n+1}$ as the weighted average 
\bq\label{eq:average}
    u^{n+1}(x) { =} \sum\limits_{i=1}^{N_d} \lambda_i(x) u_i^n(x).
\eq

For convenience, we can combine this with the subdomain solution operators into a single operator
\begin{align}\label{eq:Gdef}
    G[u] \equiv \sum\limits_{i=1}^{N_d} \lambda_i S^h_i[u].
\end{align}
Then we can write the DDM iteration as
\bq\label{eq:update} 
    u^{n+1} = G[u^n] = G^{n+1}[u_0].
\eq

\section{Convergence}\label{sec:conv}
The goal of this section is to prove the convergence of the DDM iteration defined in the previous section.  In particular, we will establish the following result on global convergence.

\begin{theorem}[Global convergence]\label{thm:mainconvergence}
Let $\G\subset\R^n$ be a finite set of points, $F^h$ be any continuous, monotone, proper approximation scheme, and $u^*$ be the unique solution of
\[ F^h(x,u(x),u(x)-u(\cdot)) = 0, \quad x \in \G.\]
Then the DDM iteration $u^{n+1} = G[u^n]$ defined by~\eqref{eq:sub2}-\eqref{eq:Gdef} converges to $u^*$ for any choice of $u^0:\G\to\R$.
\end{theorem}

\subsection{Monotonicity of DDM iteration}



We begin by establishing that the DDM iteration is itself a monotone mapping.  We first recall that the subdomain operators $F_i^h$~\eqref{eq:sub2} are trivially continuous, monotone, and proper since the underlying scheme $F^h$ has these properties.  This allows us to restate the Discrete Comparison 
Principle (Theorem \ref{thm:DCP}) in terms of the subdomain solution operators.


\begin{lemma}[Discrete comparison principle]\label{lem:subDCP}
Under the hypotheses of Theorem~\ref{thm:mainconvergence}, suppose that $u \leq v$ on $\G$.  Then
 $S^h_i[u] \leq S^h_i[v]$ on $\G$ for every $i  =1, \ldots, d$.
\end{lemma}

An immediate consequence of the subdomain map being a proper, monotone scheme is that the DDM update operator preserves inequalities on $\G$.
\begin{corollary}[Monotonicity of DDM map]\label{cor:Gmono}
Under the hypotheses of Theorem~\ref{thm:mainconvergence}, suppose that $u \leq v$ on $\G$. Then $G[u] \leq G[v]$ on $\G$.
\end{corollary}
\begin{proof}
From Lemma~\ref{lem:subDCP}, we have that 
\begin{align*}
   S^h_i[u] \leq S^h_i[v]
\end{align*}
for each $i = 1,\ldots,N_d$. 

Since the DDM update operator is given by a convex combination of these solutions operators, we have that
\begin{align*}
    G[u] = \sum\limits_{i=1}^{N_d} \lambda_i S^h_i[u] \leq \sum\limits_{i=1}^{N_d} \lambda_i S^h_i[v] = G[v].
\end{align*}
\end{proof}

\subsection{Fixed point of the DDM iteration}
In this section, we will establish that the DDM iteration has a unique fixed point, which corresponds to the unique solution $u^*$ of the underlying approximation scheme~\eqref{eq:schemetosolve}.

We begin by showing that the solution $u^*$ is a fixed point of the DDM mapping.
\begin{lemma}[Existence of fixed point]\label{lem:solnfp}
 Under the hypotheses of Theorem~\ref{thm:mainconvergence}, let $u^*$ be the unique solution to 
 \[ F^h(x,u(x),u(x)-u(\cdot)) = 0, \quad x \in \G.\]
 Then $G[u^*] = u^*$.
\end{lemma}
\begin{proof}
We first need to establish that for any $i = 1, \ldots, N_d$, $u^*$ is a fixed point of the solution operator $S_i^h[u]$.  To this end, we consider the $ith$ subdomain operator
$ F_i^h(x,u^*;u^*).$

If $x \in \G_i$, we have
\[  F^h_i(x,u^*;u^*) = F^h(x,u^*(x),u^*(x)-u^*(\cdot)) = 0.\]
Similarly, if $x \in \G\cap\partial\Omega$ we have
\[  F^h_i(x,u^*;u^*) = u^*(x)-g(x) =F^h(x,u^*(x),u^*(x)-u^*(\cdot)) = 0. \]
Finally, if $x \in \G\cap\Omega\cap\G_i^c$ we have
\[ F^h_i(x,u^*;u^*) = u^*(x)-u^*(x) = 0.\]
Since the solution operator is uniquely defined by the equation
\[ F_i^h(x,S_i^h[u^*];u^*) = 0, \quad x \in \G,\]
this establishes that
\[ S_i^h[u^*] = u^*.\]

Since the DDM operator is expressed as a convex combination of these solution operators, we can immediately compute
\begin{align*}
    G[u^*] &= \sum\limits_{i=1}^{N_d} \lambda_i S_i^h[u^*]
    \\
    & = \sum\limits_{i=1}^{N_d} \lambda_i u^*
    \\
    & = u^*.
\end{align*}
\end{proof}

Next, we need to show that any fixed point of the DDM mapping is a solution to the scheme~\eqref{eq:schemetosolve}.  We begin by stating a couple preliminary lemmas regarding the properties of fixed points.

\begin{lemma}[Behavior of fixed points on boundary]\label{lem:fixedptbdy}
Under the hypotheses of Theorem~\ref{thm:mainconvergence}, let $u$ be any fixed point of the DDM map $G[u]$.  Then $u(x)=g(x)$ for every $x \in \G\cap\partial\Omega$.
\end{lemma}
\begin{proof}
Note that each subdomain problem~\eqref{eq:sub2} enforces $S^h_i[u](x) = g(x)$ for any $x \in \G\cap\partial\Omega$ and $i = 1, \ldots, N_d$.  Then at these boundary points the fixed point satisfies
\[ u(x) = \sum\limits_{i=1}^{N_d}\lambda_i(x)S^h_i[u](x) = g(x).\]
\end{proof}

\begin{lemma}[Solution operators for fixed points]\label{lem:solnopfixed}
Under the hypotheses of Theorem~\ref{thm:mainconvergence}, let $u$ be any fixed point of the DDM map $G[u]$.  Then $S_i^h[u](x) = u(x)$ for every $x \in \G\cap\G_i^c$ and $i = 1, \ldots, d$.
\end{lemma}
\begin{proof}
These case of $x\in\partial\Omega$ is encompassed in Lemma~\ref{lem:fixedptbdy}.  Otherwise, for $x \in \G\cap\Omega\cap\G_i^c$, the subdomain problem~\eqref{eq:sub2} trivially enforces $S_i^h[u](x) = u(x)$.
\end{proof}

Now we can establish that the fixed point does solve the desired scheme~\eqref{eq:schemetosolve}.
\begin{lemma}[Fixed point is a solution]\label{lem:fptosoln}
Under the hypotheses of Theorem~\ref{thm:mainconvergence}, let $u$ be any fixed point of the DDM map $G[u]$.  Then\[F^h(x,u(x),u(x)-u(\cdot)) = 0, \quad x \in \G.\]
\end{lemma}
\begin{proof}
Suppose there exists some $y\in\G$ such that 
\[ F^h(y,u(y),u(y)-u(\cdot)) \neq 0.\]
Since the fixed point satisfies the given Dirichlet data (Lemma~\ref{lem:fixedptbdy}), this must occur at an interior point $y \in\Omega$.  This, in turn, means that $y\in\G_j$ belongs to one of the overlapping subdomains for some $j = 1, \ldots, N_d$.  Then we notice that
\[ F^h_j(y,u;u) = F^h(y,u(y),u(y)-u(\cdot)) \neq 0.\]
Thus $u$ is not a fixed point of this particular subdomain problem: $u \neq S_j^h[u]$.

In particular, there exists some $x \in\G$ such that $u(x) \neq S^h_j[u](x)$.  Without loss of generality, we may suppose that $u(x)>S^h_j[u](x)$.  From 
Lemma~\ref{lem:solnopfixed}, $x \in \G_j$.  

We recall that the fixed point can be expressed as a weighted average of these solution operators,
\[ u(x) = G[u](x) = \sum\limits_{i=1}^d\lambda_i(x) S^h_i[u](x).\]
Since $\lambda_j(x)>0$, there must exist another index $k = 1, \ldots, N_d$ such that $\lambda_k(x)>0$ and 
\bq\label{eq:ordering}S^h_k[u](x)>u(x)>S^h_j[u](x).\eq

Now we introduce the notation
\[u_i(x) = S^h_i[u](x), \quad i = 1, \ldots, N_d\] 
and consider the discrete minimization problem
\bq\label{eq:min}
u_m(z)-u_l(z) = \min\limits_{i,n=1, \ldots, N_d}\min\limits_{x \in \G}\{u_i(x)-u_n(x)\}.
\eq

From~\eqref{eq:ordering}, we note that this minimum must be negative:
\bq\label{eq:neg}
u_m(z)-u_l(z)<0.
\eq
We now consider several different possibilities for the location of $z\in\G$.

{\bf Case 1: $z \in \G_m^c\cap\G_l^c$.} By Lemma~\ref{lem:solnopfixed}, the subdomain solution operators satisfy $u_m(z) = u(z) = u_l(z)$, a contradiction.

{\bf Case 2: $z \in \G_m \cap \G_l$.} Since $u_m$ and $u_l$ solve their respective subdomain problems, we find that
\begin{align*}
    F^h(z,u_l(z),u_l(z)-u_l(\cdot)) &= F^h_l(z,u_l;u_l) = 0\\
    F^h(z,u_m(z),u_m(z)-u_m(\cdot)) &= F^h_m(z,u_m;u_m) = 0.
\end{align*}

From~\eqref{eq:min}-\eqref{eq:neg}, we have that
\bq\label{eq:moninupts}\begin{split}
u_m(z) &< u_l(z)\\
u_m(z)-u_m(x) &\leq u_l(z)-u_l(x), \quad x \in \G.
\end{split}\eq
Since $F^h$ is monotone and proper, this implies that
\[ F^h(z,u_m(z),u_m(z)-u_m(\cdot)) < F^h(z,u_l(z),u_l(z)-u_l(\cdot)),\]
which contradicts the fact that both of these operators vanish.

{\bf Case 3: $z \in \G_l \cap \G_m^c$ ($z \in \G_l^c \cap \G_m$).} We consider the first of these possibilities; the proof of the other setting is analogous.

By Lemma~\ref{lem:solnopfixed} and~\eqref{eq:moninupts}, we know that
\[ u_l(z) > u_m(z) = S^h_m[u](z) = u(z).  \]
This setting also requires $\lambda_l(z)>0$.  Given that the fixed point satisfies 
\[ u(z) = G[u](z) = \sum\limits_{i=1}^{N_d}\lambda_i(z)u_i(z),\]
there must be some other index $p \in \{1, \ldots, N_d\}$ such that $\lambda_p(z)>0$ and $u_p(z)<u(z) = u_m(z)$.

From this information, we observe that
\[ u_p(z)-u_l(z) < u_m(z)-u_l(z),\]
which contradicts the fact that $u_m(z)-u_l(z)$ is the minimum value in~\eqref{eq:min}.

We conclude that actually
\[ F^h(y,u(y),u(y)-u(\cdot)) = 0\]
for every $y \in \G$.
\end{proof}

Lemmas~\ref{lem:solnfp} and~\ref{lem:fptosoln} immediately yield the existence of a unique fixed point, which coincides with the solution of the scheme~\eqref{eq:schemetosolve}.

\begin{theorem}[Fixed point of DDM]\label{thm:fixedpt}
Under the hypotheses of Theorem~\ref{thm:mainconvergence}, the DDM mapping $G[u] = u$ has a unique fixed point $u^*$, which is given by the unique solution of
\[ F^h(x,u(x),u(x)-u(\cdot)) = 0, \quad x \in \G.\]
\end{theorem}

\subsection{Convergence}


In this section, we turn our attention to the proof of the main convergence result (Theorem~\ref{thm:mainconvergence}).  Before showing global convergence, we establish that the DDM iteration converges if initialized with a sub- or supersolution of the scheme~\eqref{eq:schemetosolve}.

\begin{definition}[Sub(super) Solution]\label{def:subsuper}
A function $u:\G\to\R$ is a sub(super)solution of the scheme $F^h$ if 
\[ F^h(x,u(x),u(x)-u(\cdot)) \leq(\geq)0, \quad x \in \G.\]
\end{definition}

\begin{theorem}[Convergence from a sub(super)solution]\label{thrm:subsupconv} 
Under the hypothesis of Theorem~\ref{thm:mainconvergence}, let $u^*$ be the unique solution of the scheme
\[ F^h(x,u(x),u(x)-u(\cdot)) = 0, \quad x \in \G.\]
If $u^0$ is a sub(super)solution of $F^h$, then the DDM iteration
 $u^{n+1} = G[u^n]$ defined in~\eqref{eq:Gdef} converges to $u^*$ as $n\to\infty$. 
\end{theorem}
\begin{proof}
We suppose without loss of generality that $u^0$ is a subsolution; the other case is analagous.  We begin by establishing that
\bq\label{eq:sandwich}
G^n[u^0] \leq G^{n+1}[u^0] \leq u^*
\eq
on $\G$ for every $n \geq 0.$

Note that the induction step follows readily from the monotonicity of the DDM mapping (Corollary~\ref{cor:Gmono}).  In particular, if~\eqref{eq:sandwich} holds then
\[ G^{n+1}[u^0] \leq G^{n+2}[u^0] \leq G[u^*]. \]
Since $u^*$ is a fixed point of the iteration (Theorem~\ref{thm:fixedpt}), the induction step follows.

To demonstrate the base case ($n=0$), we consider the $ith$ subdomain problem for any $i = 1, \ldots, N_d$.  Since $u^0$ is a subsolution of $F^h$, whenever $x\in\G_i$ or $x \in \G\cap\partial\Omega$ we have
\[ F_i^h(x,u^0;u^0) = F^h(x,u^0(x),u^0(x)-u^0(\cdot)) \leq 0.\]
On the other hand, if $x\in\G\cap\Omega\cap\G_i^c$, we have
\[ F_i^h(x,u^0;u^0) = u^0(x)-u^0(x) = 0.\]
Taking this all together, we find that
\[ F_i^h(x,u^0;u^0) \leq 0, \quad x \in \G. \]

We recall also that the $ith$ solution operator satisfies
\[ F_i^h(x,S_i^h[u^0];u^0) = 0, \quad x \in \G.\]
By the Discrete Comparison Principle (Lemma~\ref{lem:subDCP}), we conclude that
\[  u^0 \leq S_i^h[u^0]\]
on $\G$ for any $i = 1, \ldots, N_d$.

Now applying one iteration of the DDM mapping, we find that
\bq\label{eq:mon} G[u^0] = \sum\limits_{i=1}^{N_d}\lambda_iS^h_i[u^0] \geq u^0.\eq

Moreover, since $u^0$ is a subsolution of $F^h$, we can apply the Discrete Comparison Principle (Theorem~\ref{thm:DCP}) to the inequality
\[ F^h(x,u^0(x),u^0(x)-u^0(\cdot)) \leq 0 = F^h(x,u^*(x),u^*(x)-u^*(\cdot)), \quad x \in \G\]
to conclude that $u^0 \leq u^*$ on $\G$.  By the monotonicity of the DDM mapping (Corollary~\ref{cor:Gmono}), we also find that
\bq\label{eq:bound} G[u^0] \leq G[u^*] = u^*. \eq

The inequalities~\eqref{eq:mon}-\eqref{eq:bound} complete the base case
\[ u^0 \leq G[u^0] \leq u^*\]
and we conclude that~\eqref{eq:sandwich} is true.

From here, we can conclude that the DDM iteration is bounded and monotonically non-decreasing, which implies convergence to a fixed point.  From Theorem~\ref{thm:fixedpt}, the only fixed point is $u^*$.
\end{proof}

We can now leverage result to prove global convergence of the DDM iteration.

\begin{proof}[Proof of Theorem~\ref{thm:mainconvergence}]
Given any $u^0:\G\to\R$, we first notice that we can bound it from below and above by a sub- and supersolution respectively.

Let $u^- = u^0 - K$ for some constant $K>0$.  Since $F^h$ is proper, there is a constant $C>0$ such that
\begin{align*}  &F^h(x,u^0(x),u^0(x)-u^0(\cdot)) - F^h(x,u^-(x),u^-(x)-u^-(\cdot))\\ &=   F^h(x,u^0(x),u^0(x)-u^0(\cdot)) - F^h(x,u^0(x)-K,u^0(x)-u^0(\cdot))\\
&\geq CK.
\end{align*}
Thus for every $x\in\G$ we have
\[ F^h(x,u^-(x),u^-(x)-u^-(\cdot)) \leq \max\limits_{x\in\G}F^h(x,u^0(x),u^0(x)-u^0(\cdot)) - CK.\]
Taking sufficiently large $K>0$ ensures that this quantity is negative for every  $x\in\G$. In that case, $u^-$ is a subsolution.  The construction of a supersolution $u^+$ is similar.

Since
\[  u^- \leq u^0 \leq u^+,\]
we can appeal to the monotonicity of the DDM mapping (Corollary~\ref{cor:Gmono}) to conclude that
\[ G^n[u^-] \leq G^n[u^0] \leq G^n[u^+] \]
on $\G$ for every $n\in\N$.

By Theorem~\ref{thrm:subsupconv}, initializing with a sub- or supersolution yields convergence:
\[ \lim\limits_{n\to\infty} G^n[u^-] = u^* = \lim\limits_{n\to\infty} G^n[u^+].\]
We conclude that 
\[\lim\limits_{n\to\infty} G^n[u^0] = u^*,\]
as desired.
\end{proof}

\section{Numerical results}\label{sec:results}

This section is devoted to the validation of the proposed domain decomposition method for 
the Monge-Amp\`ere equation. 

\subsection{Discretization}
We begin by discretizing the \MA equation using the quadrature-based scheme proposed in~\cite{brusca2022convergent}, which is monotone and has a formal truncation error of $\bO(h^{4/3})$.  The discretization is based upon the following representation of the \MA operator as a Gaussian integral:
\bq\label{eq:MAInt}
{\det(D^2u) = \left(\frac{1}{\pi}\int_{0}^{\pi}\frac{d\theta}{u_{\theta\theta}(x)}\right)^{-2}.}
\eq

We first generate the computational domain.  Our starting point is a Cartesian mesh $\{(ih,jh) \mid i,j\in\Z\}$ that tiles $\R^2$ for some grid spacing $h>0$.  We also choose a stencil width {$w = \ceil{h^{-1/3}}$}.  Now we let $(r_j,\theta_j)$ be the polar coordinates of the following grid-aligned points lying within the desired stencil width: 
\bq\label{eq:l1angles} r_j(\cos\theta_j, \sin\theta_j) = h\left(w-j, w-\abs{w-j}\right), \quad j = 0, \ldots, 2w-1. \eq

From this tiling of $\R^2$ and angular discretization $0 = \theta_0 < \theta_1 < \ldots < \theta_{2w-1} < \pi$, we generate a set of discretization points $\G$ by (1) including all mesh points lying in the interior of the domain $\Omega$ and (2) supplementing with points in $\partial\Omega$ in order to preserve the existence of grid points perfectly aligned with the given set of angles.  That is, given any interior node $x\in\G\cap\Omega$  and $j = 0, \ldots, M$, we have
\[ x \pm r^\pm_j(x)(\cos\theta_j,\sin\theta_j) \in \mathcal{G} \]
for some $r^\pm_j(x)>0$ as in Figure~\ref{fig:stencil}.

Now for any interior node $x\in\G\cap\Omega$ and any direction $\nu_j = (\cos\theta_j,\sin\theta_j)$, we can approximate the second directional derivative of a function $u$ in the direction of $\nu_j$ by
\[ \Dt_{\nu_j\nu_j}u(x) = \frac{r^-_j(x)u(x-r^-(x)\nu_j) + r^+_j(x)u(x+r^+(x)\nu_j)-(r^+(x)+r^-(x))u(x)}{r^+(x)r^-(x)(r^+(x)+r^-(x))}. \]
Note that except in a narrow band near $\partial\Omega$, this reduces to a standard second-order centered difference discretization of the form~\eqref{eq:centered}.

We now regularize the integral in~\eqref{eq:MAInt} and discretize using a non-uniform Simpson's rule to produce a scheme of the form
\bq\label{eq:quadscheme}
\begin{split}
{F}^h(x,u(x),u(x)-u(\cdot)) =
&-\left(\frac{1}{\pi}\sum\limits_{i=0}^{2w-1} \frac{\mu_j}{\max\{\Dt_{\nu_j\nu_j}u(x), {{h^2}}\}} \right)^{-2}\\  &- \min\limits_{j=0,\ldots,2w-1}\left\{\Dt_{\nu_j\nu_j} u(x),{{h^2}}\right\}
+f(x)
\end{split}
\eq
where the quadrature weights are given by
\bq\label{eq:quadweights}
\mu_j = \begin{cases}
\dfrac{(d\theta_{j-1}+d\theta_j)^3}{6d\theta_{j-1}d\theta_j}, & j \text{ odd}\\
\dfrac{d\theta_j+d\theta_{j+1}}{6}\left(2-\dfrac{d\theta_{j+1}}{d\theta_j}\right) +\dfrac{d\theta_{j-2}+d\theta_{j-1}}{6}\left(2-\dfrac{d\theta_{j-2}}{d\theta_{j-1}}\right), & j \text{ even}.
\end{cases}
\eq
Above, $d\theta_j = \theta_{j+1}-\theta_j$ is the local angular resolution of the discretization and we define $d\theta_{2w-1} = \theta_0 + \pi - \theta_{2w-1}$.

\subsection{Implementation of DDM}

In our tests, we consider  the domain $\Omega=(-L,L)^2$.  The interior of the domain is discretized using an  $N\times N$ lattice of uniformly distributed nodes, with the addition of boundary points as described above.  
The mesh nodes are then distributed to $N_d$ subdomains in an $m$ by $n$ block format. 
Figure~\ref{fig:ddmtikz}  illustrates this decomposition of $n \times m$. 

\tikzmath{\dx = 0.2; \dy = 0.3;}
\begin{figure}
\centering
{\begin{tikzpicture}[scale=1.2]
\draw (-1,1) -- (-1,-1);
\draw (-1,-1) -- (1,-1);
\draw (1,-1) -- (1,1);
\draw (1,1) -- (-1,1);
\draw[dotted] (-1,0) -- (1,0);
\draw[dotted] (0,1) -- (0,-1);
\node at (-0.5,0.5) {$\Omega_1$};
\node at (0.5,0.5) {$\Omega_2$};
\node at (-0.5,-0.5) {$\Omega_3$};
\node at (0.5,-0.5) {$\Omega_4$};
\draw[dashed] (-1,-\dy) -- (\dx,-\dy);
\draw[dashed] (\dx,-\dy) -- (\dx,1);
\node[above] at (\dx/2,1) {$p_x$};
\node[left] at (-1,-\dy/2) {$p_y$};
\end{tikzpicture}}
\hspace{0.5in}
{\begin{tikzpicture}[scale=1.2]
\draw (-1,1) -- (-1,-1);
\draw (-1,-1) -- (1,-1);
\draw (1,-1) -- (1,1);
\draw (1,1) -- (-1,1);
\draw[dotted] (0,1) -- (0,-1);
\draw[dotted] (-1,0) -- (1,0);
\draw[dotted] (-1,0.5) -- (1,0.5);
\draw[dotted] (-1,-0.5) -- (1,-0.5);
\node at (0.5,0.25) {$\Omega_i$};
\draw[dashed] (1,0.6) -- (-0.1,0.6) -- (-0.1,-0.1) -- (1,-0.1);

\draw [decorate,decoration={brace,amplitude=3pt},xshift=-3pt,yshift=0pt]
        (-1,-1) -- (-1,1) node [black,midway,xshift=-20pt] { $n=4$};
\draw [decorate,decoration={brace,amplitude=3pt},xshift=0pt,yshift=3pt]
        (-1,1) -- (1,1) node [black,midway,yshift=9pt] { $m=2$};
\end{tikzpicture}}
\caption{Overlapping decompositions with $N_d = 2 \times 2$ and $N_d = 4\times 2$ subdomains. In general $N_d = n \times m$ means $m$ equal splits along the x-axis, and $n$ equal splits along the y-axis.}
\label{fig:ddmtikz}
\end{figure}
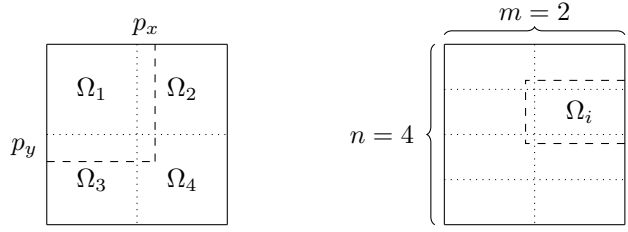

To define the overlap between subdomains, we introduce two integers  $\delta_x$ and $\delta_y$, which determine how many layers of nodes the subdomains should extend in the $x$ and $y$ directions. We will discuss the overlap amount in terms of percentages: $x$-overlap percentage $p_x=\frac{\delta_x}{N_x/n}$ and $y$-overlap percentage $p_y=\frac{\delta_y}{N_y/m}$. Figure~\ref{fig:ddmtikz} illustrates an example with the subdomain $\Omega_1$ enlarged with $p_x = 20\%$ and $p_y=30\%$. Note that 100\% overlap is the case where there is complete overlap between neighboring subdomains. By specifying the overlap as a percentage, it becomes easier to adjust and compare the overlap between different configurations. It allows for a flexible and intuitive way to control the level of interaction between subdomains and explore the trade-off between computational efficiency and accuracy in the solution. 

The subdomain problems are solved using Newton's method.  This consists of solving the non linear systems 
\[ F_i^h(x,u_i^{n+1};u^n) = 0 \]
on each subdomain $i = 1, \ldots, N_d$ in order to obtain the DDM updates $u_i^{n+1}$.  These systems are solved using a Newton-Krylov (NK) method, which involves an inner iteration of the form
\bq\label{eq:newton}
\nabla_u F_i^h(x,u_i^{n+1,k};u^n) y^k = -F_i^h(x,u_i^{n+1,k};u^n), \quad u_i^{n+1,k+1} = u_i^{n+1,k} + \lambda_k y^k,
\eq
where $\lambda_k$ is computed with the linesearch method~{\cite{Brune_2015}} and we utilize the exact analytical Jacobian. 
The Newton iteration in each subdomain is terminated when the residual satisfies {$\|F_i[u^n]\|_2<h$}.

This method requires repeated solutions of the linear system 
$$ \nabla_u F_i^h(x,u_i^{n+1,k};u^n) y^k = -F_i^h(x,u_i^{n+1,k};u^n)$$ 
for the update $y^k$, which is accomplished using a Krylov solver. 
We use deflated restarting GMRES~{\cite{erhel_restarted_1996}} for the Krylov method. 
The Krylov method is itself an iterative method, and though it will in theory converge in a finite number of steps, it is often better to truncate once some tolerance is met. We terminate the Krylov iteration when the $L^2$ norm of the relative residual is less than $10^{-5}$.

Theorem~\ref{thm:mainconvergence} guarantees that the DDM algorithm converges as the number of DDM iterations $n\ra \infty$.   Here we use the stopping criterion by $\|F^h(x,u^n(x),u^n(x)-u^n(\cdot))\|_2 < h$, where here the residual is computed over the entire domain.

The DDM iteration is initialized by first solving the problem on a coarser grid (with grid spacing $4h$), then interpolating the result onto the finer grid.  The numerical tests  are performed using  
the PETSc library~\cite{bueler2021petsc,petsc-user-ref}.

\subsection{Computational Tests}
We perform computational tests using two different examples.  

The first example involves a smooth $C^\infty$ solution of the \MA equation with data given by
\begin{equation}\tag{Ex. 1}\label{eq:ex1}
    u(\textbf{x}) = \text{exp}\bigg{(}\frac{\norm{\textbf{x}}_2^2}{2}\bigg{)},\quad f(\textbf{x}) = \bigg{(}1+\norm{\textbf{x}}_2^2\bigg{)}\text{exp}\big{(}\norm{\textbf{x}}_2^2\big{)}.
\end{equation}

The second example is a non-classical $C^1$ viscosity solution of the \MA equation.  Moreover, this example is not uniformly elliptic since the solution is convex but not strictly convex (and $f=0$ in part of the domain).  {{The solution and problem data are given by
}}
    \begin{equation}\tag{Ex. 2}\label{eq:ex5}
        u(\textbf{x}) = \text{max}\bigg{(} \norm{\textbf{x}}_2 -\frac{1}{5}  , 0 \bigg{)}^{5/2}
        ,\quad
        f(\textbf{x}) = \frac{3}{8}\text{max}\bigg{(} 5\norm{\textbf{x}}_2 -1  , 0 \bigg{)}^{2}\norm{\textbf{x}}_2^{-1}.
    \end{equation}
		
In both examples, the magnitude of the gradient grows as the distance from the origin increases.  See Figure~\ref{fig:solutions}.

\begin{figure}
    \centering
    \begin{subfigure}[]{0.45\textwidth}\includegraphics[width=\textwidth]{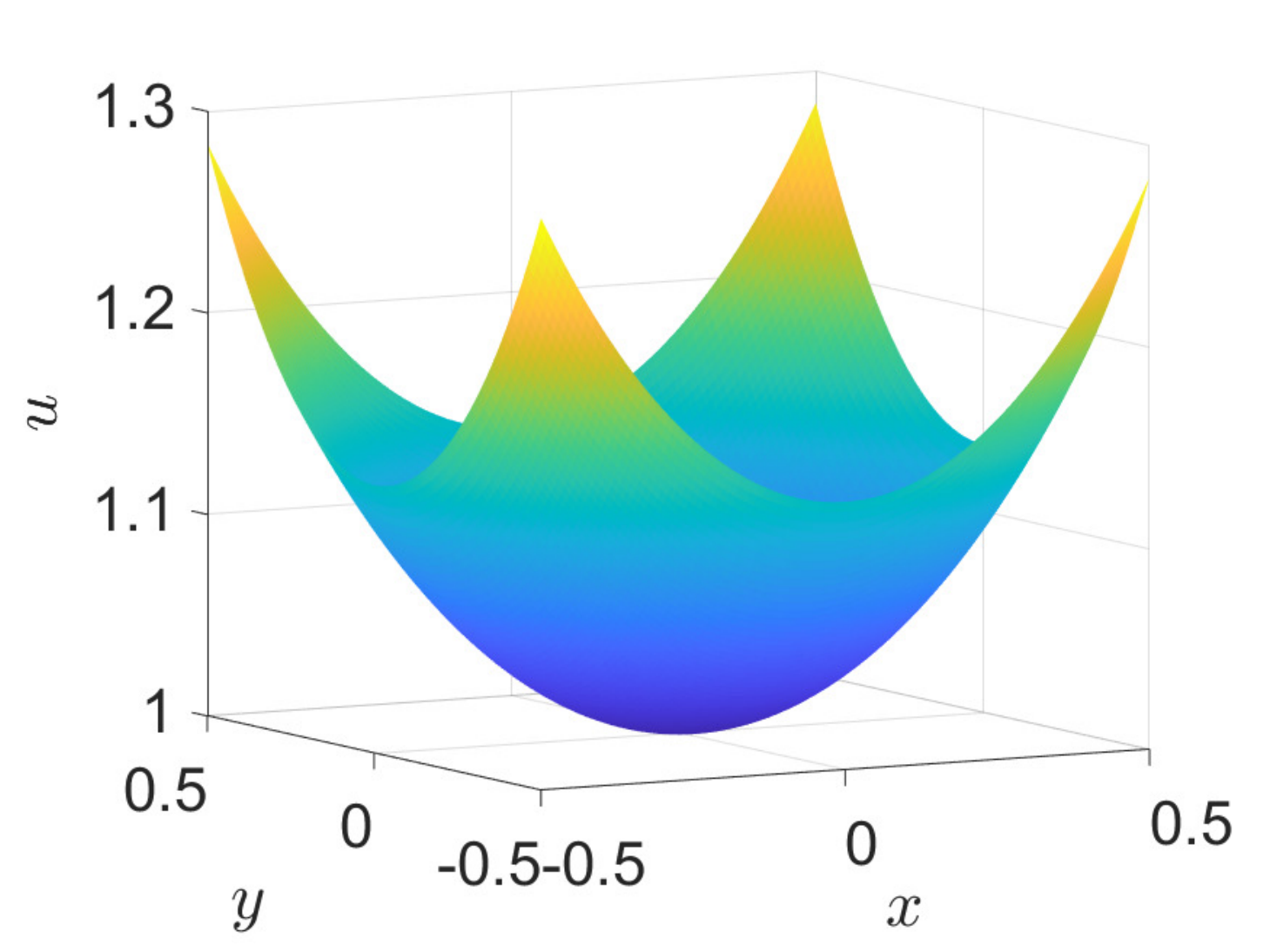}\caption{A smooth solution.}\label{fig:u1}\end{subfigure}
    \begin{subfigure}[]{0.45\textwidth}\includegraphics[width=\textwidth]{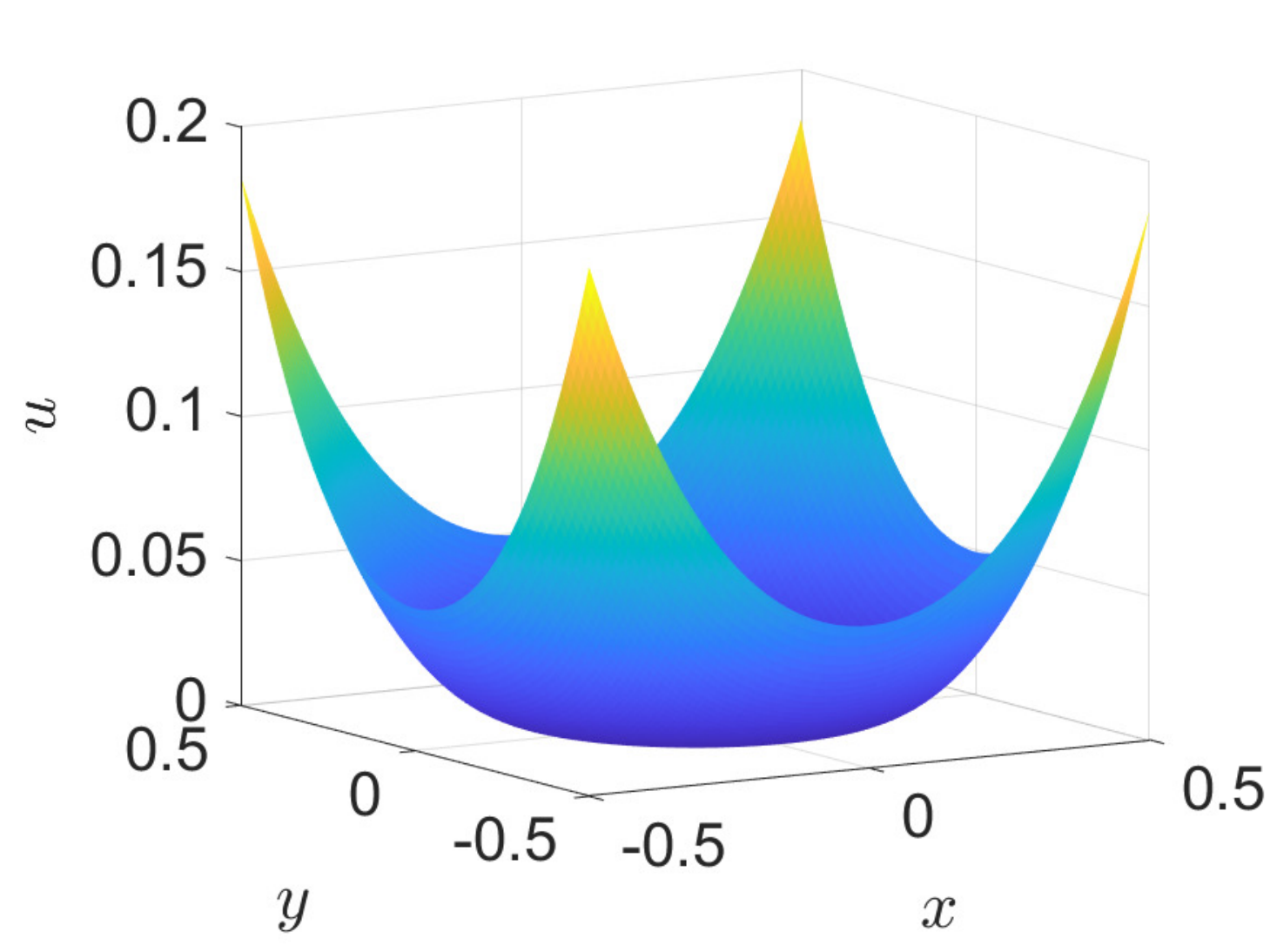}\caption{A $C^1$ solution.}\label{fig:u2}\end{subfigure}
    \caption{Exact solutions of \MA used in numerical tests.}
    \label{fig:solutions}
\end{figure}

In the first test, we provide benchmark results by solving the \MA equation using a global Newton-Krylov (NK) solver (that is, $N_d = 1$).  We provide the number of Newton iterations required 
 with respect  to the size of the original domain $L$ and the discretization parameter $h$.  Tables~\ref{tab:nk_nd1p}-\ref{tab:nks_nd4p} summarize the results.  We see that the convergence of Newton's method deteriorates with respect to the size $L$ of the domain and resolution $h$ of the grid regardless of the regularity of the solution.   We also provide the resulting error in the solutions, which is consistent with the expected consistency error.

\begin{table}[h]
        \centering
        \begin{subtable}{0.4\linewidth}
            \setlength{\tabcolsep}{4pt}       
            \renewcommand{\arraystretch}{1} 
                \begin{tabular}{@{}c|c|c@{}}
                $L$ & Iterations & $L^2$ Error  \\ \hline
                0.5 & 6 & 3.73E-04  \\
                1.0 & 7 & 4.76E-03  \\
                1.5 & 11 & 5.36E-02  \\
                2.0 & 32 & 5.28E-01  \\
                \end{tabular}
            \caption{$h = 0.05$}
            \label{st:nk_h0.05}
        \end{subtable}
        \begin{subtable}{0.4\linewidth}
            \centering
            \setlength{\tabcolsep}{4pt}       
            \renewcommand{\arraystretch}{1} 
                \begin{tabular}{@{}c|c|c@{}}
                $L$ & Iterations & $L^2$ Error  \\ \hline
                0.5 & 4 & 3.89E-05  \\
                1.0 & 7 & 5.49E-04  \\
                1.5 & 19 & 5.02E-03  \\
                2.0 & 53 & 5.57E-02 \\
                \end{tabular}
            \caption{$h = 0.01$}
            \label{st:nk_h0.01}
        \end{subtable} \\
        \caption{Results of Newton's Method for Example 1.}
        \label{tab:nk_nd1p}
    \end{table}

        \begin{table}[h]
	\vspace*{-6pt}
        \centering
        \begin{subtable}{0.4\linewidth}
            \setlength{\tabcolsep}{4pt}       
            \renewcommand{\arraystretch}{1} 
                \begin{tabular}{@{}c|c|c@{}}
                $L$ & Iterations & $L^2$ Error  \\ \hline
                0.5 & 9 & 1.61E-03  \\
                1.0 & 8 & 5.60E-03  \\
                1.5 & 12 & 1.27E-02 \\
                2.0 & 10 & 2.20E-02  \\
                \end{tabular}
            \caption{$h = 0.05$ }
            \label{st:h0.05}
        \end{subtable}
        \begin{subtable}{0.4\linewidth}
            \centering
            \setlength{\tabcolsep}{4pt}       
            \renewcommand{\arraystretch}{1} 
                \begin{tabular}{@{}c|c|c@{}}
                $L$ & Iterations & $L^2$ Error \\ \hline
                0.5 & 17 & 1.74E-04  \\
                1.0 & 17 & 5.87E-04  \\
                1.5 & 18 & 1.15E-03  \\
                2.0 & 21 & 1.98E-03 \\
                \end{tabular}
            \caption{$h = 0.01$}
            \label{st:h0.01}
        \end{subtable}
        \caption{Results of Newton's Method for Example 2.}
        \label{tab:nks_nd4p}
    \end{table}

 Let us now apply the DDM algorithm.
To test the DDM procedure, we vary  the overlap percentage $p_x, p_y$ between subdomains, which is  fixed according to the the length of subdomains.  In this work, we consider only uniform overlap ($p_x = p_y = p$).   We choose to test the iterative algorithm for $p_1= 10\%$, $p_2= 20\%$, $p_3= 30\%$, and $p_4= 40\%$. The tests are performed for domain sizes $L=0.5,1.0,1.5,2.0$ and for the discretization parameters $h=0.05,0.01$.  
 
 The results are displayed in Tables~\ref{tab:nasm}-\ref{tab:nasm2}.  In all cases, we verify that the observed solution errors obtained from the DDM solutions are identical to the errors obtained with NK alone, which depend only on the parameters $h$ and $L$, and are independent of the particular details of the DDM implementation.  In particular, this verifies that the method has successfully converged to the desired solution.

One key observation is that the number of iterations decreases as the overlap $p$ is increased.  This phenomenon holds even when $N_d$ is large.  We also find that the iteration count depends on the structure of the subdivision of the domain, with lower iteration counts when the horizontal and vertical dimensions are subdivided in the same way ($m=n$). Perhaps more interesting, we find that the iteration counts remain fairly stable as the discretization parameter $h$ decreases and the domain size $L$ increases.  This is in contrast to the iteration counts of the NK solver on the global domain, which scaled poorly with $h$ and $L$.  Moreover, it is critical to note that while each DDM iteration involves inner Newton-Krylov solves, these occur on small subdomains.  As observed in the benchmark results (Tables~\ref{tab:nk_nd1p}-\ref{tab:nks_nd4p}), the NK solve is far more efficient on smaller domains.  These results hold even for the non-smooth example.  


\begin{table}[h]
\centering
\begin{subtable}{\linewidth}
\centering
\setlength{\tabcolsep}{3pt}       
\renewcommand{\arraystretch}{1.}
\begin{tabular}{c|cccc|cccc|cccc|cccc|cccc}
&\multicolumn{4}{c}{$N_d = 2\times 1$}&\multicolumn{4}{c}{$N_d = 2\times 2$}&\multicolumn{4}{c}{$N_d = 3 \times 2$}&\multicolumn{4}{c}{$N_d = 4 \times 2$}&\multicolumn{4}{c}{$N_d = 3 \times 3$}
\\
$L$ & $p_1$ & $p_2$ & $p_3$ & $p_4$ & $p_1$ & $p_2$ & $p_3$ & $p_4$ & $p_1$ & $p_2$ & $p_3$ & $p_4$ & $p_1$ & $p_2$ & $p_3$ & $p_4$ & $p_1$ & $p_2$ & $p_3$ & $p_4$
\\ \hline
0.5 & 7 & 5 & 4 & 3 & 10 & 7 & 6 & 4 & 13 & 9 & 8 & 6 & 15 & 14 & 9 & 9 & 16 & 10 & 10 & 8 
\\
1.0 & 11 & 7 & 5 & 5 & 16 & 10 & 8 & 6 & 21 & 15 & 11 & 9 & 30 & 21 & 15 & 13 & 25 & 19 & 15 & 11
\\
1.5 & 14 & 9 & 7 & 6 & 21 & 13 & 9 & 7 &33 & 20 & 14 & 11 & 38 & 26 & 19 & 15 & 43 & 27 & 18 & 15
\\
2.0 & 18 & 11 & 8 & 7 & 26 & 15 & 11 & 9 & 38 & 22 & 17 & 13 & 60 & 33 & 23 & 19 & 49 & 28 & 22 & 17
\end{tabular}
\caption{$h = 0.05$}
\label{st:nasm_h0.05}
\end{subtable}

\vspace*{10pt}
\begin{subtable}{\linewidth}
\centering
\setlength{\tabcolsep}{3pt}       
\renewcommand{\arraystretch}{1.}
\begin{tabular}{c|cccc|cccc|cccc|cccc|cccc}
&\multicolumn{4}{c}{$N_d = 2\times 1$}&\multicolumn{4}{c}{$N_d = 2\times 2$}&\multicolumn{4}{c}{$N_d = 3 \times 2$}&\multicolumn{4}{c}{$N_d = 4 \times 2$}&\multicolumn{4}{c}{$N_d = 3 \times 3$}
\\
$L$ & $p_1$ & $p_2$ & $p_3$ & $p_4$ & $p_1$ & $p_2$ & $p_3$ & $p_4$ & $p_1$ & $p_2$ & $p_3$ & $p_4$ & $p_1$ & $p_2$ & $p_3$ & $p_4$ & $p_1$ & $p_2$ & $p_3$ & $p_4$
\\ \hline
0.5 & 16 & 9 & 6 & 5 & 24 & 13 & 9 & 7 & 29 & 18 & 13 & 10 & 41 & 26 & 18 & 15 & 35 & 23 & 17 & 13
\\
1.0 & 19 & 10 & 7 & 6 & 30 & 16 & 11 & 8 & 40 & 22 & 16 & 13 & 58 & 33 & 23 & 18 & 51 & 28 & 21 & 17
\\
1.5 & 22 & 12 & 8 & 7 & 34 & 18 & 12 & 9 & 49 & 26 & 19 & 15 & 68 & 38 & 27 & 22 & 64 & 35 & 25 & 19
\\
2.0 & 30 & 14 & 12 & 10 & 38 & 20 & 14 & 12 & 55 & 31 & 22 & 17 & 81 & 45 & 32 & 25 & 73 & 41 & 29 & 22
\end{tabular}
\caption{$h = 0.01$}
\label{st:nasm_h0.01}
\end{subtable}
\caption{Number of DDM iterations for Example 1 as a function of the number of subdomains $N_d$, overlap $p$, domain size $L$, and grid spacing $h$.}
\label{tab:nasm}
\end{table}

\begin{table}[h]
\centering
\begin{subtable}{\linewidth}
\centering
\setlength{\tabcolsep}{3pt}       
\renewcommand{\arraystretch}{1.}
\begin{tabular}{c|cccc|cccc|cccc|cccc|cccc}
&\multicolumn{4}{c}{$N_d = 2\times 1$}&\multicolumn{4}{c}{$N_d = 2\times 2$}&\multicolumn{4}{c}{$N_d = 3 \times 2$}&\multicolumn{4}{c}{$N_d = 4 \times 2$}&\multicolumn{4}{c}{$N_d = 3 \times 3$}
\\
$L$ & $p_1$ & $p_2$ & $p_3$ & $p_4$ & $p_1$ & $p_2$ & $p_3$ & $p_4$ & $p_1$ & $p_2$ & $p_3$ & $p_4$ & $p_1$ & $p_2$ & $p_3$ & $p_4$ & $p_1$ & $p_2$ & $p_3$ & $p_4$
\\ \hline
0.5 & 6 & 4 & 4 & 3 & 8 & 6 & 5 & 4 & 10 & 7 & 7 & 5 & 12 & 12 & 8 & 8 & 13 & 8 & 8 & 6
\\
1.0 & 11 & 7 & 5 & 4 & 15 & 10 & 8 & 6 & 20 & 14 & 11 & 8 & 29 & 21 & 15 & 13 & 24 & 18 & 14 & 11
\\
1.5 & 14 & 8 & 6 & 5 & 20 & 12 & 9 & 7 & 31 & 19 & 13 & 11 & 36 & 25 & 18 & 15 & 42 & 26 & 17 & 14
\\
2.0 & 16 & 9 & 7 & 5 & 23 & 14 & 10 & 8 & 34 & 19 & 15 & 12 & 54 & 29 & 21 & 17 & 43 & 25 & 20 & 15
\end{tabular}
            \caption{$h = 0.05$}
            \label{st:nasm2_h0.05}
        \end{subtable}

\vspace*{10pt}
\begin{subtable}{\linewidth}
\centering
\setlength{\tabcolsep}{3pt}       
\renewcommand{\arraystretch}{1.}
\begin{tabular}{c|cccc|cccc|cccc|cccc|cccc}
&\multicolumn{4}{c}{$N_d = 2\times 1$}&\multicolumn{4}{c}{$N_d = 2\times 2$}&\multicolumn{4}{c}{$N_d = 3 \times 2$}&\multicolumn{4}{c}{$N_d = 4 \times 2$}&\multicolumn{4}{c}{$N_d = 3 \times 3$}
\\
$L$ & $p_1$ & $p_2$ & $p_3$ & $p_4$ & $p_1$ & $p_2$ & $p_3$ & $p_4$ & $p_1$ & $p_2$ & $p_3$ & $p_4$ & $p_1$ & $p_2$ & $p_3$ & $p_4$ & $p_1$ & $p_2$ & $p_3$ & $p_4$
\\ \hline
0.5 & 13 & 8 & 6 & 5 & 18 & 11 & 8 & 7 & 23 & 15 & 11 & 9 & 33 & 22 & 15 & 13 & 29 & 19 & 14 & 11
\\
1.0 & 18 & 10 & 7 & 6 & 27 & 15 & 10 & 8 & 38 & 21 & 15 & 12 & 54 & 31 & 22 & 17 & 49 & 27 & 20 & 16
\\
1.5 & 20 & 11 & 8 & 6 & 31 & 17 & 11 & 9 & 46 & 24 & 17 & 14 & 63 & 35 & 25 & 20 & 60 & 32 & 24 & 18
\\
2.0 & 20 & 11 & 8 & 6 & 32 & 17 & 12 & 9 & 48 & 26 & 19 & 15 & 69 & 38 & 28 & 22 & 63 & 36 & 26 & 20 
\end{tabular}
\caption{$h = 0.01$}
\label{st:nasm2_h0.01}
\end{subtable}
\caption{Number of DDM iterations for Example 2 as a function of the number of subdomains $N_d$, overlap $p$, domain size $L$, and grid spacing $h$.}
\label{tab:nasm2}
\end{table}

  
The last test is dedicated to a large domain size ($L=2$) with an increasing number of subdomains $N_d=n\times m$ arranged symmetrically with $n=m$.  We perform this test using $h=0.01$. See Figure~\ref{fig:resultsL2}.  While the total number of subdomains $N_d$ grows quadratically in this figure, the number of DDM iterations increases only linearly.  This suggests that the cost of DDM will scale very well with an increasing number of subdomains and parallelization.  A larger overlap percentage in the DDM results in fewer iterations, even for very large values of $N_d$. This result aligns with the expectation that increased overlap facilitates better communication, resulting in faster convergence.

    \begin{figure}[h]
        \centering
        \includegraphics[width=0.5\linewidth]{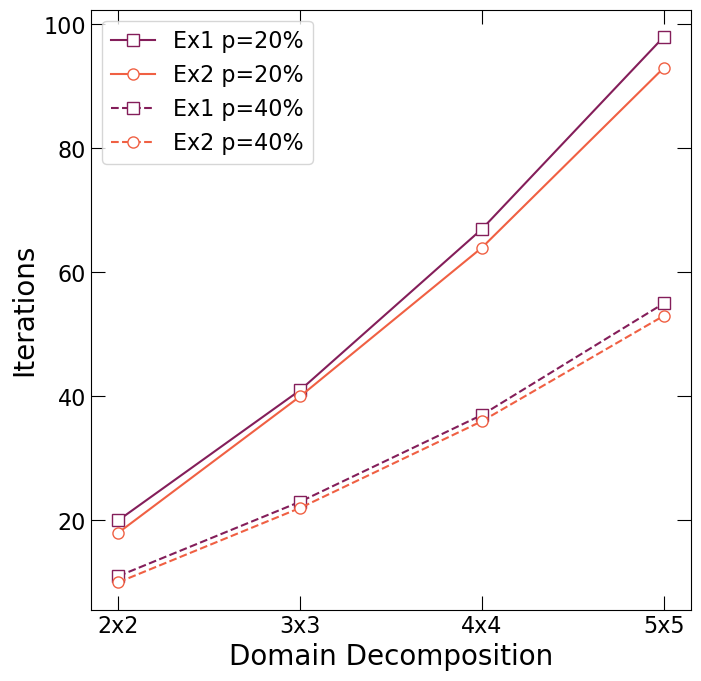}
        \caption{Number of DDM iterations as a function of the number of subdomains $m=n$ along each dimension.}
        \label{fig:resultsL2}
    \end{figure}

\section{Conclusion}\label{sec:conclusion}
{ {
In this paper, we have introduced a new domain decomposition method  for solving  the \MA equation.  We  showed that this method can be appropriately 
coupled with wide stencil approximations needed to ensure convergence to the weak solution of the equation. Using a discrete comparison principle argument, we have proved the convergence of the resulting iterative method to the solution of the underlying discrete scheme.  This proof establishes global convergence, given an arbitrary initial guess, and applies to any consistent and monotone discretization of the \MA equation.

We have validated our algorithm on examples of varying regularity.  These experiments confirm the proof of convergence presented in this paper. Moreover, the computational experiments demonstrate iteration counts that are fairly stable with respect to variations in problem size and discretization parameters, which is often not the case for standard Newton solvers applied to such strongly nonlinear problems.  Each iteration require the solution of several small sub-problems, which can be accomplished in parallel and with significantly less cost than a single global Newton update.  These observations continue to hold even for examples with less regularity and a loss of uniform ellipticity. The obtained iterative solver can further be improved by using adequate preconditioners and subdivision of the original computational domain, optimized utilization of parallel resources, 
and lower-tolerance solution to the subproblems.
The DDM solver appears very promising for solving this kind of equations.  We expect that even more dramatic improvements will be evident as we proceed to higher-dimensions, more singular problems, more highly resolved grids, and better optimized solvers. 

}}

\bibliographystyle{plain}
\bibliography{MABib}

\end{document}